\documentclass[a4paper, 11pt, English]{article}
\usepackage[utf8]{inputenc}
\usepackage[T1]{fontenc}
\usepackage{graphicx}
\usepackage{stmaryrd}
\usepackage[a4paper]{geometry}
\usepackage{amsmath,amsfonts,amssymb,amsthm,epsfig,epstopdf,url,array}
\usepackage{rotating}
\usepackage[colorlinks=true,citecolor=red,linkcolor=blue,pdfpagetransition=Blinds]{hyperref}
\usepackage{cleveref}
\usepackage{nameref}
\usepackage{enumitem}
\usepackage{comment}
\Crefname{paragraph}{Section}{Sections}
\usepackage{fancyhdr}

\usepackage{fullpage}

\newcommand{\ensemblenombre}[1]{\mathbb{#1}}
\newcommand{\N}{\ensemblenombre{N}}

\newcommand{\R}{} % Probleme LateXML
\renewcommand{\R}{\ensemblenombre{R}}

\newcommand{\norme}[1]{\left\lVert#1\right\rVert}

\newcommand{\diam}{\mathrm{diam}}
\newcommand{\sign}{\mathrm{sign}}

\newcommand{\dive}[1]{\mathrm{div}}
\theoremstyle{plain} 
\newtheorem{prop}{Proposition}[section] 
\newtheorem{theo}[prop]{Theorem}

\newtheorem{lem}[prop]{Lemma}

\theoremstyle{definition}

\newtheorem{rmk}[prop]{Remark}

\newtheorem{conj}[prop]{Conjecture}

\makeatletter
\let\original@addcontentsline\addcontentsline
\newcommand{\dummy@addcontentsline}[3]{}
\newcommand{\DeactivateToc}{\let\addcontentsline\dummy@addcontentsline}
\newcommand{\ActivateToc}{\let\addcontentsline\original@addcontentsline}
\makeatother

\begin{document}

\title{Exponential bounds for gradient of solutions to linear elliptic and parabolic equations}
\author{Kévin Le Balc'h}

\maketitle
\begin{abstract}
In this paper, we prove global gradient estimates for solutions to linear elliptic and parabolic equations. For a sufficiently smooth bounded convex domain $\Omega \subset \R^N$, we show that a solution $\phi \in W_0^{1,\infty}(\Omega)$ to an appropriate elliptic equation $\mathcal{L} \phi = F$, with $F \in L^{\infty}(\Omega;\R)$, satisfies $|\nabla \phi|_{\infty} \leq C |F|_{\infty}$, with a positive constant $C = \exp(C(\mathcal{L})\diam(\Omega))$. We also obtain similiar estimates in the parabolic setting. The proof of these exponential bounds relies on global gradient estimates inspired by a series of papers by Ben Andrews and Julie Clutterbuck. This work is motivated by a dual version of the Landis conjecture.
\end{abstract}
\small
\tableofcontents
\normalsize

\section{Introduction: The Landis conjecture as a motivation }

In the late 1960s, see \cite{KL88}, Evgeni Landis conjectured the following result.
\begin{conj}[“Qualitative” Landis conjecture]\label{conj:Landis}
For a potential $V \in L^{\infty}(\R^N)$, if $u$ is a solution to the Schrödinger elliptic equation 
\begin{equation}
\label{eq:ellipticSchro}
- \Delta u + V(x) u = 0\ \text{in}\ \R^N,
\end{equation}
and $u$ satisfies the exponential decay estimate
\begin{equation}
\label{eq:expdecayIntro}
\exists C, \varepsilon > 0,\ \forall x \in \R^N,\ |u(x)| \leq C \exp(-|x|^{1+\varepsilon}),
\end{equation}
then $$u \equiv 0.$$
\end{conj}
\indent The Landis' conjecture was latter disproved by Viktor Meshkov in 1991, see \cite{Mes91}, who constructed non-trivial complex-valued functions $u$ and $V$, satisfying \eqref{eq:ellipticSchro} in $\R^2$, with $V$ bounded and such that $|u(x)| \leq C \exp(-|x|^{4/3})$. He also proved the following qualitative unique continuation result: for $N\geq 1$, if $u$ satisfies \eqref{eq:ellipticSchro} with $V \in L^{\infty}(\R^N)$ and $|u(x)| \leq C \exp(-|x|^{4/3+\varepsilon})$ then $u \equiv 0$.\\
\indent The conjecture has been brought back to attention in $2000's$ by the works of Jean Bourgain and Carlos Kenig, see \cite{BK05}. In \cite{Ken06}, the author improves Meshkov's qualitative unique continuation result in the case of real-valued functions, pushing the decay condition up to $|u(x)| \leq C \exp(-|x|^{4/3} \log(|x|))$. However, there is not an analogue of Meshkov's counterexample for real-valued $u$ and $V$. This fact led Carlos Kenig to ask in \cite[Question 1]{Ken06} whether, for real-valued $u$ and $V$ satisfying \eqref{eq:ellipticSchro}, the condition \eqref{eq:expdecayIntro} necessarily implies $u \equiv 0$. This conjecture is still open nowadays, except for particular situations that we describe in the next paragraphs.\\
\indent In \cite{KSW15}, Carlos Kenig, Luis Silvestre and Jenn-Nan Wang prove “quantitative” Landis conjecture in the plane $\R^2$, assuming that $V \geq 0$. Observe that in this case the “qualitative” Landis conjecture is actually trivial by applying the maximum principle. The main contribution of \cite{KSW15} is the derivation of quantitative unique continuation for non-trivial solution $u$ to \eqref{eq:ellipticSchro}: If $|u(z)| \leq \exp( C |z|)$ then
\begin{equation}
\label{eq:quantitativeContinuation}
\exists R_0 > 0,\ \forall R \geq R_0,\ \inf_{|z_0|=R} \sup_{|z-z_0| < 1} |u(z)| \geq \exp (-C R \log(R)).
\end{equation}
Actually, they prove \eqref{eq:quantitativeContinuation} for solution to more general elliptic operators i.e. $\mathcal{L}u = -\Delta u - \nabla \cdot (W u) + V u$, with $V, W$ bounded, real-valued and $V \geq 0$. Note that in this case, $\mathcal{L}$ does not satisfy the maximum principle a priori, see \cite[ Chapter 8, Section 8.1]{GT83}. The article \cite{KSW15} actually leads to a series of papers \cite{KW15, DZ18, DZ19, Dav20, DW20}.\\
\indent In \cite{Ros18}, for a general elliptic operator $\mathcal{L}$, Luca Rossi proves the Landis conjecture with a sharp rate of decay in the following cases: $N=1$, or $u$ is radial or $\mathcal{L}$ is radial or $u \geq 0$ or $\lambda_1(\mathcal{L}) \geq 0$ where $\lambda_1(\mathcal{L})$ is the generalized principal eigenvalue of $\mathcal{L}$, see also the article \cite{ABG17} for a probabilistic proof to some of these previous results.\\
\indent In the one-dimensional case, we present an elementary proof of the Landis conjecture, relying on some kind of duality argument, communicated by Michel Pierre. This proof actually motivates our present work.
\begin{theo}
\label{th:LandisOneD}
Let $W, V \in L^{\infty}(\R;\R)$ and $u$ be a solution to 
\begin{equation}
\label{equ1D}
-u'' - (Wu) ' + V u = 0,
\end{equation}
satisfying the exponential decay \eqref{eq:expdecayIntro}. Then $u \equiv 0$.
\end{theo}
\begin{proof}
\indent We introduce the function $\sign$
\begin{equation}
\forall s \in \R,\ \sign(s) = s\ \text{if}\ s \geq 0,\ -s \ \text{if}\ s < 0.
\end{equation}
\indent For the proof, we assume that $W$ and $V$ are continuous bounded functions but the arguments can be easily adapted to the more general case $W, V \in L^{\infty}(\R;\R)$.\\
\indent Let $R>0$ and let us solve the following evolution problem of second order
\begin{equation}
\label{eq:OdeSecond}
- \phi'' +  W \phi' + V \phi = \sign(u),\ \phi(-R) = \phi'(-R) = 0.
\end{equation}
By Cauchy-Lipschitz theorem, the solution to \eqref{eq:OdeSecond} is global.\\
\indent We rewrite \eqref{eq:OdeSecond} as an ordinary differential equation of first order
\begin{equation}
\label{eq:EDOPhi}
\Phi' = M \Phi + \Theta,\ \Phi(-R) = 0,
\end{equation}
where 
\begin{equation}
\Phi = \begin{pmatrix}
\phi\\ \phi'
\end{pmatrix}
,\ M = M(x) = \begin{pmatrix}
0& 1\\
V(x) & W(x)
\end{pmatrix},\
\Theta = \begin{pmatrix}
0\\ \sign(u)
\end{pmatrix}
\end{equation}
By Gronwall's estimate applied to \eqref{eq:EDOPhi}, we show that there exists $C>0$ such that
\begin{equation}
\label{eq:Estiphiphi'}
\forall x \in [-R, +\infty),\ |\Phi(x)| = |\phi(x)| + |\phi'(x)| \leq C \exp( C |x|).
\end{equation}
\indent Let $1 \leq x_1 \leq x_2 \in \R$. By using \eqref{equ1D}, \eqref{eq:expdecayIntro}, we have
\begin{align}
\left|u'(x_2) - u'(x_1) \right|= \left|\int_{x_1}^{x_2} u''\right| &=\left| \int_{x_1}^{x_2} - (W u)'(x) + V(x) u(x) dx \right|\notag\\
& \leq |W(x_2) u(x_2) - W(x_1) u(x_1)| + C \left|\int_{x_1}^{x_2} x e^{-x^{1+\varepsilon}} dx \right|\notag\\
& \leq |W(x_2) u(x_2) - W(x_1) u(x_1)| + C\left( e^{-x_1^{1+\varepsilon}} - e^{-x_2^{1+\varepsilon}}\right).\label{eq:CauchyCriteria}
\end{align}
This implies that $u'(x)$ admits a limit near $+ \infty$ by Cauchy's criteria and this limit is necessary $0$ because $\lim_{x \rightarrow + \infty} u(x) = 0$. Moreover, by passing to the limit $x_2 \rightarrow + \infty$ in \eqref{eq:CauchyCriteria}, we find that
\begin{equation}
\label{eq:Estiu'}
\forall x \in [1,+\infty),\ |u'(x)| \leq C \exp(-|x| ^{1+\varepsilon}).
\end{equation}
So by multiplying \eqref{eq:OdeSecond} by $u$ then integrate between $(-R,R)$ and integrate by parts, using \eqref{equ1D}, we find
\begin{align}
\int_{-R}^{R} |u| &= -\phi'(R) u(R) + \phi'(-R) u(-R) + \phi(R) u'(R) - \phi(-R) u'(-R) \notag \\
&+ W(R) u(R) \phi(R) - W(-R) u(-R) \phi(-R).
\label{eq:boundaryintegrate}
\end{align}
So by letting $R \rightarrow + \infty$ in \eqref{eq:boundaryintegrate}, using $\phi(-R) = \phi'(-R) = 0$, \eqref{eq:Estiphiphi'} and \eqref{eq:Estiu'}, we get
\begin{equation}
\int_{-\infty}^{+\infty} |u| = 0.
\end{equation}
Then $u \equiv 0$.
\end{proof}
The crucial point in the previous proof is the derivation of exponential bounds for the solution $\phi$ and its derivative $\phi'$ of the adjoint equation \eqref{eq:OdeSecond}. They are obtained in \eqref{eq:Estiphiphi'} by using classical Gronwall's estimate, remarking that in the one dimensional case, an elliptic operator can always be seen as an evolution operator.\\
\indent The main goal of the paper is to see to what extent we can extend these $L^{\infty}$ exponential bounds for solutions to elliptic equations and parabolic equations in the multi-dimensional case, see \Cref{th:mainresult1} and \Cref{th:mainresult2} below.

\section{Main results}

\subsection{Statement of the main results}

%Notation: for $0 < r < R$, $B_R$ will stand for the open ball of center $0$ and radius $R$, $A_{r,R}$ will stand for the open annulus of center $0$ of small radius $r$ and big radius $R$.
Let $\Omega$ be a smooth open bounded strictly convex domain of $\R^N$ with $N \geq 1$. We will use the notation $\diam(\Omega)$ for the diameter of $\Omega$, $\diam(\Omega) = \sup\{|y-x|\ ;\ x, y \in \Omega\}$.

The first main result of the paper is an estimation of the growth of the gradient of a solution to a linear elliptic equation, posed in $\Omega$, with homogeneous Dirichlet boundary conditions.

\begin{theo}
\label{th:mainresult1}
Let $W \in L^{\infty}(\Omega;\R^N)$, $V \in L^{\infty}(\Omega;\R^{+})$ and $F \in L^{\infty}(\Omega;\R)$. Let $\phi$ be the real-valued solution to the boundary elliptic problem
\begin{equation}
\label{eq:ellipticphi}
\left\{
\begin{array}{l l}
-\Delta \phi + W \cdot \nabla \phi + V \phi= F &\mathrm{in}\ \Omega,\\
\phi = 0&\mathrm{on}\ \partial \Omega.
\end{array}
\right.
\end{equation}
Then there exists an universal positive constant $C>0$ depending on $N$ such that 
\begin{equation}
\label{eq:GradElliptic}
\forall x \in \Omega,\ |\nabla \phi(x)| \leq \exp\Bigg(C\Big(1 + \norme{W}_{L^{\infty}} + \norme{V}_{L^{\infty}}^{1/2}\Big) \diam(\Omega)\Bigg) \norme{F}_{L^{\infty}} .
\end{equation}
\end{theo}
Before continuing, let us make some comments on \Cref{th:mainresult1}.
\begin{itemize}
\item For homogeneous Neumann boundary conditions on $\phi$, we can also prove \Cref{th:mainresult1}, but only for potential $V= 0$.
\item By using the homogeneous Dirichlet boundary conditions for $\phi$, it is easy to deduce from \eqref{eq:GradElliptic} the following estimate
$$\norme{ \phi}_{W_0^{1, \infty}(\Omega)}\leq \exp\Bigg(C\Big(1 + \norme{W}_{L^{\infty}} + \norme{V}_{L^{\infty}}^{1/2}\Big) \diam(\Omega)\Bigg) \norme{F}_{L^{\infty}} .$$
\item Let us remark that \eqref{eq:ellipticphi} is well-posed because of the sign condition on the potential $V$, according to \cite[Chapter 8, Section 8.2]{GT83}. From maximal elliptic regularity in $L^p$, $1 < p < +\infty$, see \cite[Chapter 9, Section 9.5, Theorem 9.11]{GT83} applied to the elliptic equation \eqref{eq:ellipticphi}, we already know that there exists a positive constant $C = C(\Omega, \norme{W}_{L^{\infty}},\norme{V}_{L^{\infty}}p)>0$ such that
\begin{equation*}
|\phi|_{W^{2,p}(\Omega)} \leq C |F|_{L^{\infty}(\Omega)}.
\end{equation*}
So by taking $p$ sufficiently large, using $W^{2,p}(\Omega) \hookrightarrow W^{1, \infty}(\Omega)$, see \cite[Chapter 7, Section 7.11]{GT83}, we obtain for some other positive constant $C = C(\Omega, \norme{W}_{L^{\infty}}, \norme{V}_{L^{\infty}})>0$, 
\begin{equation*}
|\phi|_{W^{1,\infty}(\Omega)} \leq C |F|_{L^{\infty}(\Omega)}.
\end{equation*}
The main interest of \eqref{eq:GradElliptic} is to make the constant $C$ explicit in function of the data of the problem, i.e. $\Omega$, $W$ and $V$. Moreover, the elliptic operator in \eqref{eq:ellipticphi} is of order two while $W \cdot \nabla \phi$ and $V \phi$ are respectively of order one and zero so \eqref{eq:GradElliptic} seems to be the good scaling as an ordinary differential argument would suggest.
\item When one look at the expected $L^{\infty}$-bound on the gradient of $\phi$, i.e. \eqref{eq:GradElliptic}, where $\phi$ is a solution to an elliptic equation, one can try to use Bernstein estimates, see \cite[Chapter 15, Section 15.1]{GT83}. Indeed, roughly speaking, Bernstein's idea consists in observing that the gradient of a solution to an elliptic equation satisfies also an elliptic equation itself then one can apply classical maximum principle. Unfortunately, this strategy can only be employed for sufficiently smooth coefficients in the elliptic equation, which is not a priori the case in \eqref{eq:ellipticphi}.
\item For $V \in L^{\infty}(\Omega;\R)$ without sign condition, if $\phi$ satisfies \eqref{eq:ellipticphi}, proving an exponential bound as \eqref{eq:GradElliptic} is actually an interesting open problem.
\end{itemize}

The second main result of the paper is an estimation of the growth of the gradient of a solution to a linear parabolic equation, posed in the parabolic cylinder $Q_T := (0,T)\times \Omega$ for $T>0$, with homogeneous Dirichlet boundary conditions on the parabolic boundary $\Sigma_T := (0,T)\times\partial\Omega$.

\begin{theo}
\label{th:mainresult2}
Let $W \in L^{\infty}(\Omega;\R^N)$, $V \in L^{\infty}(\Omega;\R)$, $F \in L^{\infty}(Q_T;\R)$ and $\phi_0 \in W_0^{1, \infty}(\Omega)$. Let $\phi$ be the real-valued solution to the parabolic equation
\begin{equation}
\label{eq:parabolicphi}
\left\{
\begin{array}{l l}
\partial_t \phi -\Delta \phi + W \cdot \nabla \phi + V \phi= F &\mathrm{in}\ Q_T,\\
\phi = 0&\mathrm{on}\ \Sigma_T,\\
\phi(0,\cdot) = \phi_0 &\mathrm{in}\ \Omega.
\end{array}
\right.
\end{equation}
Then there exists an universal positive constant $C>0$ depending on $N$ such that for every $t \geq 0$, 
\begin{align}
\notag
&\norme{\nabla \phi(t,\cdot)}_{\infty}\\
&\leq \exp\Bigg(C\left(T \norme{V}_{\infty}+\Big(1 + \norme{W}_{\infty} + \norme{V}_{\infty}^{1/2}\Big) \diam(\Omega)\right)\Bigg) \bigg(\norme{\nabla \phi_0}_{\infty} +  \norme{F}_{\infty} \bigg).\label{eq:GradParabolic}
\end{align}
\end{theo}
Before continuing, let us make some comments on \Cref{th:mainresult2}.
\begin{itemize}
\item For $V=0$, we can also prove \Cref{th:mainresult2} for $W \in L^{\infty}(Q_T;\R^N)$. We expect that the result of \Cref{th:mainresult2} is also true for $W \in L^{\infty}(Q_T;\R^N), V \in L^{\infty}(Q_T;\R)$, see \Cref{rmk:multiplierdt} below.
\item For homogeneous Neumann boundary conditions on $\phi$, we can also prove \Cref{th:mainresult2}, but only for potential $V= 0$.
\item As for the elliptic case, the constant appearing in the estimate \eqref{eq:GradParabolic} seems to have the good scaling in time and space. Indeed, the first term $\exp(CT\norme{V}_{\infty})$ is due to the natural dissipation in time of the differential operator $\partial_t + V$ while the second term comes from the fact that the elliptic operator $ -\Delta + W \cdot \nabla + V$ is of order two in space and $W \cdot \nabla$, $V$ are respectively perturbations of order one and order zero.
\item For $F=0$ and $V=0$, the estimate \eqref{eq:GradParabolic} has to be linked to the classical $L^2$ dissipation estimate
\begin{equation}
\label{eq:expL2}
\norme{ \phi(T,\cdot)}_{L^2(\Omega)} +\norme{\nabla \phi}_{L^{2}(Q_T)}  \leq \exp\Bigg(C  \norme{W}_{L^{\infty}}^2 T \Bigg) \norme{\phi_0}_{L^{2}(\Omega)},
\end{equation}
obtained by multiplication of \eqref{eq:parabolicphi} by $\phi$, Young's inequality and Gronwall's lemma. First, the estimate \eqref{eq:expL2} is not as accurate as \eqref{eq:GradParabolic} for large values of  $\norme{W}_{L^{\infty}}$. Secondly and more surprising, the estimate \eqref{eq:GradParabolic} does not depend on the size of the time interval $(0,T)$.
\item Comparing to the elliptic case, we do not impose sign condition on the potential $V$ because we can always assume that $V \geq 0$, considering the parabolic equation satisfied by $\exp(-t |V|_{\infty}) \phi$.
\end{itemize}

\subsection{Strategy of the proofs and organization of the paper}

To prove \Cref{th:mainresult1}, \Cref{th:mainresult2}, we use gradient estimates inspired by a series of paper from Andrews and Clutterbuck, see \cite{AC09a, AC09b} and the nice survey \cite{And12} for an introduction to such a technique. Basically, the crucial idea is to remark that the modulus of continuity of the solution to an elliptic equation, respectively a parabolic equation, is also a solution to an elliptic equation, respectively parabolic equation itself. So, applying maximum principle, one can deduce pointwise bounds on the modulus of continuity therefore pointwise bounds on the gradient. This is done in \Cref{sec:expgrowthelliptic} for the elliptic case and in \Cref{sec:expgrowthparabolic} for the parabolic case.\\
\indent In \Cref{sec:weakquantitativeelliptic}, we discuss some applications of \Cref{th:mainresult1} to weak quantitative unique continuation in $L^1$ for elliptic equations. To do this, we use a duality argument similar to the one performed in the proof of \Cref{th:LandisOneD}.

\section{Proof of the main results}

\subsection{Proof of the exponential bound for solution to elliptic equation}
\label{sec:expgrowthelliptic}
The goal of this section is to prove \Cref{th:mainresult1}.\\
\indent We first prove \Cref{th:mainresult1} in the special case $V = 0$ and homogeneous Neumann boundary conditions. Secondly, we make a small adaptation for obtaining \Cref{th:mainresult1} in the case $V=0$ and homogeneous Dirichlet boundary conditions. Finally, we prove \Cref{th:mainresult1} in the case $V \geq 0$ and homogeneous Dirichlet boundary conditions. 
\begin{proof}[Proof of \Cref{th:mainresult1} for $V= 0$ and the homogeneous Neumann boundary case.]
We denote $R := \diam(\Omega)$.\\
\indent Let $\phi$ be the solution of \eqref{eq:ellipticphi}, with homogeneous Neumann boundary conditions and with $V=0$.\\
\indent Let us set
\begin{equation}
\label{eq:defZ}
\forall (x, y) \in \overline{\Omega} \times \overline{\Omega},\ Z(x,y) := \phi(y) - \phi(x) - 2 \varphi\left(\frac{|y-x|}{2}\right),
\end{equation}
where $\varphi$ is the solution to the following ordinary differential equation
\begin{equation}
\label{eq:ellipticvarphi}
\left\{
\begin{array}{l l}
- \varphi'' = (|W|_{\infty}+1) \varphi' + 2 |F|_{\infty} &\mathrm{in}\ (0,+\infty),\\
\varphi(0) = 0,\ \varphi'(0) = \lambda,
\end{array}
\right.
\end{equation}
with $\lambda >0$ that will be fixed later to ensure that 
\begin{equation}
\label{eq:signvarphi'}
\varphi'>0\ \text{in}\ (0, R).
\end{equation}.
\indent We split the proof in three main steps.\\

\indent \textbf{Step 1: $Z$ is nonpositive}.\\
\indent The goal of this step is to prove that 
\begin{equation}
\label{eq:Zneg}
\forall (x,y) \in \overline{\Omega} \times \overline{\Omega},\ Z(x,y) \leq 0.
\end{equation}
\indent The function $Z$ is continuous on the compact set $\overline{\Omega}^2$, so admits a maximum attained at $(x_0, y_0) \in \overline{\Omega}^2$.\\
\indent If $x_0 = y_0$ then $Z(x_0,y_0) \leq 0$ because $\varphi(0)=0$ so \eqref{eq:Zneg} holds.\\
\indent If $y_0 \in \partial \Omega$ then
\begin{equation*}
D_{\nu_{y_0}} Z(x_0,y_0) =D_{\nu_{y_0}} \phi(y_0) - \varphi' \frac{(y_0-x_0)}{|y_0-x_0|} \cdot \nu_{y _0} \leq - \varphi' \frac{(y_0-x_0)}{|y_0-x_0|} \cdot \nu_{y _0} < 0,
\end{equation*}
because of the homogeneous Neumann boundary conditions, the strict convexity of $\Omega$, implying that $\frac{(y_0-x_0)}{|y_0-x_0|} \cdot \nu_{y _0} < 0$ and $\varphi'>0$. We then deduce that $Z(x_0, y_0 - s \nu_{y_0}) > Z(x_0,y_0)$ for $s>0$ sufficiently small, which contradicts the definition of the maximum.\\
\indent If $x_0 \in \partial \Omega$, the same arguments work.\\
\indent Now we can assume that $(x_0, y_0)$ belong to the open set  $\Omega \times \Omega$. By optimality conditions, we have
\begin{equation*}
\partial_x Z(x_0,y_0 ) = \partial_{y} Z(x_0,y_0)=0 \ \text{and}\ \text{Hess}(Z)(x_0,y_0)\ \text{is nonpositive}.
\end{equation*}
In particular, we can choose an orthonormal basis $(e_i)_{1 \leq i \leq n}$ with $e_n = \frac{y_0-x_0}{|y_0-x_0|}$ and 
\begin{align}
&\frac{d^2}{ds^2} Z(x_0+se_n, y_0-se_n)_{|s=0} \leq 0, \label{eq:Derivativenn}\\
&\frac{d^2}{ds^2} Z(x_0+se_i, y_0+se_i)_{|s=0} \leq 0,\ i=1, \dots, n-1.\label{eq:Derivativeii}
\end{align}
The vanishing of the first derivatives gives
\begin{equation}
\label{eq:vanishderivative}
\nabla \phi (y_0) = \varphi' e_n,\ \nabla \phi(x_0) = \varphi' e_n.
\end{equation}
Along the path $(\hat{x}, \hat{y})=(x_0+s e_i, y_0 + s e_i)$, the distance $|\hat{y}-\hat{x}|$ is constant, so for $i\neq n$, by \eqref{eq:Derivativeii}, we have
\begin{equation}
\label{eq:Hessi}
0 \geq \frac{d^2}{ds^2} Z(x_0+se_i, y_0+se_i)_{|s=0} = \partial_{i}\partial_{i} \phi(y_0) - \partial_{i}\partial_{i} \phi(x_0).
\end{equation}
Along the path $(\hat{x}, \hat{y})=(x_0+s e_n, y_0 - s e_n)$, we have $\frac{d}{ds}|\hat{y}-\hat{x}| = -2$ and $\frac{d^2}{ds^2} |\hat{y}-\hat{x}|= 0$ so by \eqref{eq:Derivativenn}, we have
\begin{equation}
\label{eq:Hessn}
0 \geq \frac{d^2}{ds^2} Z(x_0+se_n, y_0-se_n)_{|s=0} = \partial_{n}\partial_{n} \phi(y_0) - \partial_{n}\partial_{n} \phi(x_0) - 2 \varphi''.
\end{equation}
We sum \eqref{eq:Hessi} for $1 \leq i \leq n-1$ and we add \eqref{eq:Hessn}, then we use \eqref{eq:ellipticphi}, \eqref{eq:ellipticvarphi}, \eqref{eq:vanishderivative} and \eqref{eq:signvarphi'} to get
\begin{align*}
0 &\geq \Delta \phi(y_0) - \Delta \phi(x_0) - 2 \varphi''\\
& \geq -F(y_0) + W(y_0) \cdot \nabla \phi(y_0) + F(x_0) - W(x_0) \cdot \nabla \phi(x_0) + 2 (|W|_{\infty}+1) \varphi' + 4 |F|_{\infty}\\
& \geq - 2 |F|_{\infty} + (W(y_0) - W(x_0)) \cdot \varphi' e_n +  2(|W|_{\infty}+1) \varphi' + 4 |F|_{\infty}\\
& \geq 2 |F|_{\infty} - 2 |W|_{\infty} \varphi' + 2 (|W|_{\infty}+1) \varphi'\\
& > 0.
\end{align*}
This is a contradiction. So $(x_0, y_0) \notin \Omega \times \Omega$.\\
\indent Therefore, the maximum is attained for $x_0 = y_0$ then $Z \leq 0$ which is the conclusion of Step 1.\\

\textbf{Step 2: Explicit form of $\varphi$.}\\
\indent From \eqref{eq:ellipticvarphi}, an easy computation leads to
\begin{equation}
\exists A,B\in \R,\ \forall x \in [0,+\infty),\ \varphi(x) = A + B e^{-(|W|_{\infty} + 1) x} - 2 \frac{|F|_{\infty}x}{(|W|_{\infty} + 1)}.
\end{equation}
The condition $\varphi(0) = 0$ implies $A + B = 0$. The condition \eqref{eq:signvarphi'} gives
\begin{equation*}
\forall x \in [0,R],\ \varphi'(x) = - B (|W|_{\infty} + 1) e^{-(|W|_{\infty} + 1) x} - 2 \frac{|F|_{\infty}}{(|W|_{\infty} + 1)}  > 0
\end{equation*}
It is sufficient to ensure that $\varphi'(R) > 0$, so
\begin{equation*}
- B (|W|_{\infty} + 1) e^{-(|W|_{\infty} + 1) R} >  2 \frac{|F|_{\infty}}{(|W|_{\infty} + 1)},
\end{equation*}
This last condition is fulfilled for instance with
\begin{equation*}
B = - 4 \frac{|F|_{\infty}}{(|W|_{\infty} + 1)^2} e^{(|W|_{\infty} + 1) R},
\end{equation*}
therefore by taking
\begin{equation}
\label{eq:defvarphi0}
\varphi'(0) = \lambda =  2 \frac{|F|_{\infty}}{(|W|_{\infty} + 1)} e^{(|W|_{\infty} + 1) R},
\end{equation}
the condition \eqref{eq:signvarphi'} is satisfied.

\textbf{Step 3: Gradient estimate.}\\
\indent From Step 1, we have \eqref{eq:Zneg}, i.e. $Z$ stays nonpositive where $Z$ is defined in \eqref{eq:defZ}. So by letting $y \rightarrow x$, we deduce from \eqref{eq:defvarphi0} and $\varphi(0)=0$
\begin{equation*}
\forall x \in \Omega,\ |\nabla \phi(x)| \leq \varphi'(0) \leq C \norme{F}_{L^{\infty}}  e^{C (\norme{W}_{L^{\infty}}+1) R},
\end{equation*}
which is exactly the desired bound \eqref{eq:GradElliptic}.
\end{proof}
By looking at the previous proof of \Cref{th:mainresult1} when $V=0$ and for homogeneous Neumann boundary conditions, we see that the homogeneous Dirichlet case can be treated in exactly the same way except when $x_0 \in \partial \Omega$ or $y_0 \in \partial \Omega$. So we only treat this case in the following proof.
\begin{proof}[Proof of \Cref{th:mainresult1} for $V= 0$ and the homogeneous Dirichlet boundary case.] We split the proof in two steps.\\
\indent \textbf{Step 1: A ponctual bound for $\phi$.} First, we show that 
\begin{equation}
\label{eq:ponctboundphivarphi}
\forall x \in \Omega,\ |\phi(x)| \leq \varphi(d(x)),
\end{equation} where $d(x) = \text{dist}(x, \partial \Omega)$ and $\varphi$ is defined in \eqref{eq:ellipticvarphi}.\\
\indent Let us fix $y \in \partial \Omega$, we set $v(x) = \varphi((x-y) \cdot \nu(y))$. We have
\begin{align*}
- \Delta v = - \nabla \cdot (\nu(y)\cdot \varphi') = - \sum_{i,j} \nu_i \nu_j \varphi''
& = \sum_{i,j} \nu_i \nu_j ((|W|_{\infty} + 1) \varphi' + |F|_{\infty}) \\
& \geq |W|_{\infty} \varphi' + |F|_{\infty},
\end{align*}
and
\begin{align*}
W \cdot \nabla v = \sum_{i} W_i(x) \cdot \nu_i(y) \varphi'.
\end{align*}
So, by using \eqref{eq:signvarphi'}, and Cauchy-Schwarz inequality,
\begin{equation*}
- \Delta v + W \cdot \nabla v \geq |W|_{\infty} \varphi' + |F|_{\infty} - \left(\sum_i W_i(x)^2\right)^{1/2} \varphi' \geq |W|_{\infty} \varphi' + |F|_{\infty} - |W|_{\infty} \varphi'  \geq 0.
\end{equation*}
Moreover, we have $v(x) = 0$ on the boundary, so $v$ is a supersolution to the elliptic problem \eqref{eq:ellipticphi} then by the comparison principle, we have
$$ \forall x \in \Omega,\   |\phi(x)| \leq v(x) = \varphi((x-y) \cdot \nu(y)).$$
We then deduce the expected bound \eqref{eq:ponctboundphivarphi} by minimizing over $y$.\\
\indent \textbf{Step 2: The maximum of $Z$ is attained at $(x_0,y_0)$ with $x_0 \in \partial \Omega$ or $y_0 \in \partial \Omega$.}  If for instance $x_0 \in \partial\Omega$, we have $\phi(x_0) = 0$ and by using \eqref{eq:ponctboundphivarphi}
\begin{align*}
Z(x_0,y_0)& = \phi(y_0) - \phi(x_0) - 2 \varphi\left(\frac{|y_0-x_0|}{2}\right)\\
& \leq \varphi(d(y_0)) - 2 \varphi\left(\frac{|y_0-x_0|}{2}\right)\\
& \leq \varphi(|y_0-x_0|) - 2 \varphi\left(\frac{|y_0-x_0|}{2}\right)\\
& \leq 0.
\end{align*}
Here, we have used the fact that $\varphi$ is increasing and concave by \eqref{eq:ellipticvarphi}.
\end{proof}
Now, we show how the proof of \Cref{th:mainresult1} for $V \geq 0$ can be reduced to the case $V = 0$. This is done thanks to the following lemma, stating the existence of a positive multiplier.
\begin{lem}
\label{lem:multiplier}
Let $W \in L^{\infty}(\Omega;\R^N)$ and $V \in L^{\infty}(\Omega;\R^+)$. Then there exist a positive constant $C>0$ and $\psi$ satisfying
\begin{align}
\label{eq:ellipticpsi}- \Delta \psi + W \cdot \nabla \psi + V \psi = 0 \ \text{in}\ \Omega,&\\
\label{eq:ponctualpsi}\exp(-C \diam(\Omega)(|W|_{\infty} + |V|_{\infty}^{1/2})) \leq \psi \leq \exp(C \diam(\Omega) (|W| _{\infty} + |V|_{\infty}^{1/2}))\ \text{in}\ \Omega,&\\
 \label{eq:ponctuallogpsi}\norme{\nabla \log(\psi)}_{L^{\infty}(\Omega)} \leq C(|W| _{\infty} + |V|_{\infty}^{1/2}).&
\end{align}
\end{lem}
Let us admit \Cref{lem:multiplier} for the moment and let us prove \Cref{th:mainresult1} for $V \geq 0$.
\begin{proof}[Proof of \Cref{th:mainresult1} for $V \geq 0$ and homogeneous Dirichlet boundary conditions.]
Let $\phi$ be the solution of \eqref{eq:ellipticphi} with homogeneous Dirichlet boundary conditions. Let $\psi$ be as in \Cref{lem:multiplier}. Then a straightforward computation, using \eqref{eq:ellipticphi} and \eqref{eq:ellipticpsi} leads to 
\begin{equation}
\label{eq:ellipticphihat}
\left\{
\begin{array}{l l}
-\Delta \widehat{\phi} + \widehat{W} \cdot \nabla \widehat{\phi} = \widehat{F} &\mathrm{in}\ \Omega,\\
 \widehat{\phi} = 0&\mathrm{on}\ \partial \Omega,
\end{array}
\right.
\end{equation}
where we have set
\begin{equation}
 \widehat{\phi} := \frac{\phi}{\psi},\  \widehat{W} := W - 2 \nabla \log(\psi),\  \widehat{F} := \frac{F}{\psi}.
\end{equation}
So by applying \Cref{th:mainresult1} for $V=0$ and using \eqref{eq:ponctualpsi}, \eqref{eq:ponctuallogpsi}, we obtain
\begin{align}
\forall x \in \Omega,\ |\nabla  \widehat{\phi}(x)| &\leq \exp\Bigg(C\Big(1 + \norme{\widehat{W}}_{L^{\infty}}\Big) \diam(\Omega)\Bigg) \norme{\widehat{F}}_{L^{\infty}}\notag\\
& \leq \exp\Bigg(C\Big(1 + \norme{W}_{L^{\infty}} + \norme{V}_{L^{\infty}}^{1/2}\Big) \diam(\Omega)\Bigg) \norme{F}_{L^{\infty}} .
\label{eq:GradElliptichat}
\end{align}
Using Dirichlet boundary conditions for $\widehat{\phi}$ and \eqref{eq:GradElliptichat}, we easily get
\begin{align}
\forall x \in \Omega,\ | \widehat{\phi}(x)| 
& \leq \exp\Bigg(C\Big(1 + \norme{W}_{L^{\infty}} + \norme{V}_{L^{\infty}}^{1/2}\Big) \diam(\Omega)\Bigg) \norme{F}_{L^{\infty}} .
\label{eq:Elliptichat}
\end{align}
By coming back to the variable $\phi$, 
$$ \nabla \phi = \widehat{\phi} \nabla \psi + \psi \nabla \widehat{\psi} = \widehat{\phi} \psi \nabla (\log \psi) + \psi \nabla \widehat{\phi},$$
using \eqref{eq:Elliptichat}, \eqref{eq:GradElliptichat}, \eqref{eq:ponctualpsi}, \eqref{eq:ponctuallogpsi}, we deduce
the expected bound \eqref{eq:GradElliptic}.
\end{proof}
Now, we devote the rest of the section to the proof of \Cref{lem:multiplier}, which is crucially inspired by \cite{KSW15}.
\begin{proof}[Proof of \Cref{lem:multiplier}]
Without loss of generality, we can assume that $\Omega = B(0,R)$ where $R = \diam(\Omega)$. We set $K = \norme{W}_{L^{\infty}(\Omega)}$ and $M = \norme{V}_{L^{\infty}\Omega)}$.\\
\indent First, we extend $W$ and $V$ by $0$ on $B(0,2R) \setminus B(0,R)$.\\
\indent We construct a positive multiplier $\psi$ as follows.\\
\indent Let $\psi_1(x) = \exp((K+\sqrt{M})  x_1)$ then 
\begin{align}
-\Delta \psi_1 + W(x) \cdot \nabla \psi_1 + V(x) \psi_1 &= (-(K + \sqrt{M})^2 + W_1(x) (K + \sqrt{M}) + V(x) ) \psi_1\notag\\
& \leq  (-(K + \sqrt{M})^2 + K (K+ \sqrt{M}) + M) \psi_1\notag\\
& \leq -K \sqrt{M} \psi_1 \notag\\
-\Delta \psi_1 + W(x) \cdot \nabla \psi_1 + V(x) \psi_1 &\leq 0\label{eq:subsolution} 
\end{align}
So $\psi_1$ is a subsolution of \eqref{eq:ellipticpsi}. On the other hand, $\psi_2(x) := \exp(2 R (K+\sqrt{M}))$ satisfies
\begin{equation}
\label{eq:supersolution}
-\Delta \psi_2+ W(x) \cdot \nabla \psi_2 +  V(x) \psi_2 = V(x) \psi_2 \geq 0\ \text{because}\ V \geq 0.
\end{equation}
So, $\psi_2$ is supersolution of \eqref{eq:ellipticpsi}. Moreover, we have
\begin{equation}
\label{eq:classsubup}
\forall x \in B(0,2R),\ \psi_1(x) \leq \psi_2(x).
\end{equation}
Therefore, from \eqref{eq:subsolution}, \eqref{eq:supersolution} and \eqref{eq:classsubup}, there exists a solution $\psi$ to \eqref{eq:ellipticpsi}. Moreover, we have that
$$ \forall x \in B(0,2R), \ \psi_1(x) \leq \psi(x) \leq \psi_2(x),$$
so we easily deduce the bound \eqref{eq:ponctualpsi}.\\
\indent Now, we set $\tilde{\psi}(x) =\psi(R x)$ and we prove the following result
\begin{lem}
\label{lem:kenig} 
For every $x \in B(0,7/5)$,
\begin{equation}
\label{eq:lemkenig}
|\nabla \log(\tilde{\psi})| \leq C R (K+ \sqrt{M}).
\end{equation}
\end{lem}
\Cref{lem:kenig} is similar to \cite[Lemma 2.2]{KSW15} for $N=2$, $W=0$ and replacing $V \leftarrow R^2 V$. For the sake of completeness, we give a new proof of \Cref{lem:kenig} in \Cref{sec:appendixa}, communicated by Carlos Kenig, valid for any spatial dimension $N$.\\
\indent From \Cref{lem:kenig}, the definition of $\tilde{\psi}$, we easily obtain \eqref{eq:ponctuallogpsi}, which concludes the proof of \Cref{lem:multiplier}.
\end{proof}

\subsection{Proof of the exponential bound for solution to parabolic equation}
\label{sec:expgrowthparabolic}

We first prove \Cref{th:mainresult2} in the special case $V = 0$ then we show how we can indeed restrict ourselves to this particular case.\\
\indent We prove \Cref{th:mainresult2} in the homogeneous Neumann boundary case then we show how to adapt the proof for the homogeneous Dirichlet boundary case.
\begin{proof}[Proof of \Cref{th:mainresult2} in the case $V=0$ and for homogeneous Neumann boundary conditions.]
For any $\varepsilon >0$, we define a function $Z_{\varepsilon}$ on $[0,T] \times \overline{\Omega} \times \overline{\Omega}$ by
\begin{equation}
Z_{\varepsilon}(t,x,y) = \phi(t,y) - \phi(t,x) - 2 \varphi\left(\frac{|y-x|}{2}\right) - \varepsilon e^t,
\end{equation}
where $\varphi$ is the solution to the ordinary differential equation \eqref{eq:ellipticvarphi} with $\lambda >0$ that has to fixed later to ensure that $\varphi' > |\nabla \phi_0|_{\infty}$.\\
\indent By assumption, $Z_{\varepsilon}(0,x,y) \leq - \varepsilon$ for every $x \neq y$ using the mean-value theorem  and $Z_{\varepsilon}(t,x,x) \leq - \varepsilon$ for every $x \in \Omega$ and $t \geq 0$. We will prove for any $\varepsilon >0$ that $Z_{\varepsilon}$ is negative on $[0,T] \times \overline{\Omega} \times \overline{\Omega}$. If this is not true, then there exists a first time $t_0 >0$ and points $x_0 \neq y_0 \in \overline{\Omega}$ such that $Z_{\varepsilon}(t_0, x_0,y_0) = 0$.\\
\indent We consider two possibilities: if $y_0 \in \partial\Omega$ then we have 
\begin{equation}
D_{\nu_{y_0}} Z_{\varepsilon}(t_0,x_0,y_0) =D_{\nu_{y_0}} \phi(t_0,y_0) - \varphi' \frac{(y_0-x_0)}{|y_0-x_0|} \cdot \nu_{y _0} \leq - \varphi' \frac{(y_0-x_0)}{|y_0-x_0|} \cdot \nu_{y _0} < 0,
\end{equation}
because of the homogeneous Neumann boundary conditions, the strict convexity of $\Omega$, implying that $\frac{(y_0-x_0)}{|y_0-x_0|} \cdot \nu_{y _0} < 0$ and $\varphi'>0$. We then deduce that $Z_{\varepsilon}(t_0,x_0, y_0 - s \nu_{y_0}) > 0$ for $s>0$ sufficiently small, which contradicts the fact that $Z_{\varepsilon} \leq 0$ on $[0,t_0] \times \times \overline{\Omega} \times \overline{\Omega}$. The case where $x_0 \in \partial\Omega$ is similar.\\
\indent The second possibility is that $x$ and $y$ are in the interior of $\Omega$. Then all first spatial derivatives of $Z_{\varepsilon}$ (in $x$ and $y$) at $(t_0, x_0, y_0)$ vanish, and the full $2 N \times 2 N$ matrix of second derivatives is nonpositive. In particular, we choose an orthonormal basis
$(e_i)_{1 \leq i \leq n}$ with $e_n = \frac{y_0-x_0}{|y_0-x_0|}$ and 
\begin{align}
&\frac{d^2}{ds^2} Z(t_0,x_0+se_n, y_0-se_n)_{|s=0} \leq 0,\\
&\frac{d^2}{ds^2} Z(t_0,x_0+se_i, y_0+se_i)_{|s=0} \leq 0,\ i=1, \dots, n-1.
\end{align}
The vanishing of the first derivatives gives
\begin{equation}
\nabla \phi (t_0,y_0) = \varphi' e_n,\ \nabla \phi(t_0,x_0) = \varphi' e_n.
\end{equation}
Along the path $(\hat{x}, \hat{y})=(x+s e_i, y + s e_i)$, the distance $|\hat{y}-\hat{x}|$ is constant, so for $i\neq n$,
\begin{equation}
\label{eq:HessiPar}
0 \geq \frac{d^2}{ds^2} Z(t_0,x_0+se_i, y_0+se_i)_{|s=0} = \partial_{i}\partial_{i} \phi(t_0,y_0) - \partial_{i}\partial_{i} \phi(t_0,x_0).
\end{equation}
Along the path $(\hat{x}, \hat{y})=(x+s e_n, y - s e_n)$, we have $\frac{d}{ds}|\hat{y}-\hat{x}| = -2$ and $\frac{d^2}{ds^2} |\hat{y}-\hat{x}|= 0$ so
\begin{equation}
\label{eq:HessnPar}
0 \geq \frac{d^2}{ds^2} Z(t_0,x_0+se_n, y_0-se_n)_{|s=0} = \partial_{n}\partial_{n} \phi(y_0) - \partial_{n}\partial_{n} \phi(t_0,x_0) - 2 \varphi''.
\end{equation}
We sum for $1 \leq i \leq n-1$ \eqref{eq:Hessi} and we add \eqref{eq:Hessn} to get
\begin{equation}
0 \geq \Delta \phi(t_0,y_0) - \Delta \phi(t_0,x_0) - 2 \varphi''.
\end{equation}
Finally, we compute the time derivative of $Z_{\varepsilon}$ at $(t_0,x_0,y_0)$
\begin{align*}
\partial_{t} Z_{\varepsilon} &=  \Delta \phi(t_0,y_0) -  W(t_0,y_0) \cdot \nabla \phi(t_0, y_0)- \Delta \phi(t_0,x_0) + F(t_0,y_0)\\
&\quad \quad + W(t_0,x_0) \cdot \nabla \phi(t_0,x_0) - F(t_0,x_0) - \varepsilon e^{t}\\
& \leq 2 \varphi'' + 2 |W|_{\infty} \varphi' + 2 | F|_{\infty} -  \varepsilon e^{t}\\
& < 0.
\end{align*}
This strict inequality is impossible since this the first time where $Z_{\varepsilon} \geq 0$. This contradiction proves that $Z_{\varepsilon} < 0$ for every $\varepsilon > 0$ and therefore $Z_0 \leq 0$.\\
\indent By looking at the end of the proof of \Cref{th:mainresult1} in the case $V=0$ and for homogeneous Neumann boundary conditions, see  Step 2 and Step 3, we easily obtain that taking
$$ \varphi'(0) = \lambda = 2 \frac{|F|_{\infty}}{(|W|_{\infty} + 1)} e^{(|W|_{\infty} + 1) R} + 2 |\nabla \phi_0|_{\infty} e^{(|W|_{\infty} + 1) R},$$
this ensures that $\varphi' > |\nabla \phi_0|_{\infty}$. Hence, we deduce the expected bound \eqref{eq:GradParabolic} which concludes the proof of \Cref{th:mainresult2}.
\end{proof}
By looking at the previous proof of \Cref{th:mainresult2} when $V=0$ and for homogeneous Neumann boundary conditions, we see that the homogeneous Dirichlet case can be treated in exactly the same way except when $x_0 \in \partial \Omega$ or $y_0 \in \partial \Omega$. So we only treat this case in the following proof.
\begin{proof}[Proof of \Cref{th:mainresult2} for $V= 0$ and the homogeneous Dirichlet boundary case.] We split the proof in two steps.\\
\indent \textbf{Step 1: A ponctual bound for $\phi$.} First, we show that 
\begin{equation}
\label{eq:ponctboundphivarphiPar}
\forall (t,x) \in [0,T]\times\Omega,\ |\phi(t,x)| \leq \varphi(d(x)),
\end{equation} where $d(x) = \text{dist}(x, \partial \Omega)$ and $\varphi$ is defined in \eqref{eq:ellipticvarphi}.\\
\indent Let us fix $y \in \partial \Omega$, we set $v(x) = \varphi((x-y) \cdot \nu(y))$. We have
\begin{align*}
- \Delta v = - \nabla \cdot (\nu(y)\cdot \varphi') = - \sum_{i,j} \nu_i \nu_j \varphi''
& = \sum_{i,j} \nu_i \nu_j ((|W|_{\infty} + 1) \varphi' + |F|_{\infty}) \\
& \geq |W|_{\infty} \varphi' + |F|_{\infty},
\end{align*}
and
\begin{align*}
W \cdot \nabla v = \sum_{i} W_i(x) \cdot \nu_i(y) \varphi'.
\end{align*}
So, by using \eqref{eq:signvarphi'}, and Cauchy-Schwarz inequality,
\begin{equation*}
- \Delta v + W \cdot \nabla v \geq |W|_{\infty} \varphi' + |F|_{\infty} - \left(\sum_i W_i(x)^2\right)^{1/2} \varphi' \geq |W|_{\infty} \varphi' + |F|_{\infty} - |W|_{\infty} \varphi'  \geq 0.
\end{equation*}
Moreover, we have $v(x) = 0$ on the boundary. Finally, we remark that 
$$ \forall x \in \Omega,\ v(x) \geq |\nabla \phi_0| |x-y| \geq  |\nabla \phi_0| d(x) \geq |\phi_0(x)|,$$
because $\phi_0 \equiv 0$ on the boundary. So $v$ is a supersolution to the parabolic problem \eqref{eq:parabolicphi} then by the comparison principle, we have
$$ \forall x \in \Omega,\   |\phi(t,x)| \leq v(x) = \varphi((x-y) \cdot \nu(y)).$$
We then deduce the expected bound \eqref{eq:ponctboundphivarphiPar} by minimizing over $y$.\\
\indent \textbf{Step 2: The maximum of $Z_{\varepsilon}$ is attained at $(t_0,x_0,y_0)$ with $x_0 \in \partial \Omega$ or $y_0 \in \partial \Omega$.}  If for instance $x_0 \in \partial\Omega$, we have $\phi(t_0,x_0) = 0$ and by using \eqref{eq:ponctboundphivarphiPar}
\begin{align*}
Z_{\varepsilon}(t_0,x_0,y_0)& = \phi(t_0,y_0) - \phi(t_0,x_0) - 2 \varphi\left(\frac{|y_0-x_0|}{2}\right) - \varepsilon e^t\\
& \leq \varphi(d(y_0)) - 2 \varphi\left(\frac{|y_0-x_0|}{2}\right)\\
& \leq \varphi(|y_0-x_0|) - 2 \varphi\left(\frac{|y_0-x_0|}{2}\right)\\
& \leq 0.
\end{align*}
Here, we have used the fact that $\varphi$ is increasing and concave by \eqref{eq:ellipticvarphi}.
\end{proof}
The proof of \Cref{th:mainresult2} for any bounded potential $V \in L^{\infty}(\Omega;\R)$ and $W \in L^{\infty}(\Omega;\R^N)$ can be reduced to the case $V = 0$, first by multipliying by $\exp(-t |V|_{\infty})$ (to reduce to the case $V \geq 0$) then by using the existence of a positive (elliptic) multiplier stated in \Cref{lem:multiplier}, see the end of the proof of \Cref{th:mainresult1}.
\begin{rmk}
\label{rmk:multiplierdt}
Actually, as mentioned before, we strongly believe that \Cref{th:mainresult2} is also true for $W \in L^{\infty}(Q_T;\R^N)$ and $V \in L^{\infty}(Q_T;\R)$. In order to obtain this result, the good strategy seems to first reduce to the case $V \geq 0$ then construct a positive multiplier satisfying
\begin{align}
\label{eq:parabolicpsi} \partial_t \psi - \Delta \psi + W \cdot \nabla \psi + V \psi = 0,\text{in}\ Q_T,&\\
\label{eq:ponctualpsiPar} \exp(-C \diam(\Omega)(|W|_{\infty} + |V|_{\infty}^{1/2})) \leq \psi \leq \exp(C \diam(\Omega) (|W| _{\infty} + |V|_{\infty}^{1/2}))\ \text{in}\ Q_T,&\\
 \label{eq:ponctuallogpsiPar}\norme{\nabla \log(\psi)}_{L^{\infty}(Q_T)} \leq C (|W| _{\infty} + |V|_{\infty}^{1/2}).&
\end{align}
\end{rmk}

\section{Weak quantitative unique continuation results for elliptic equations}
\label{sec:weakquantitativeelliptic}

\indent By looking at the proof of \Cref{th:LandisOneD}, it is natural to discuss the implications of \Cref{th:mainresult1} to the Landis conjecture. We can establish new weak quantitative unique continuation in $L^1$ for elliptic operators $\mathcal{L}u =  - \Delta u - \nabla \cdot (W u) + V u$ with $W \in L^{\infty}(\R^N;\R^N)$, $V \in L^{\infty}(\R^N;\R^{+})$.

We have the following result.
\begin{prop}
\label{cor:Landis}
Let $W \in L^{\infty}(\R^N;\R^N)$, $V \in L^{\infty}(\R^N;\R^+)$ and $u$ be a (smooth) real valued solution of the following elliptic equation
\begin{equation}
\label{eq:ellipticu}
\begin{array}{l l}
-\Delta u - \nabla \cdot (W u) + V u = 0 &\mathrm{in}\ \R^N.
\end{array}
\end{equation}
Then there exists a universal positive constant $C>0$ such that 
\begin{equation}
\label{eq:quantitativeL1}
\forall R > 0,\ \int_{|x| < R} |u(x)| dx \leq \exp\Bigg(C\Big(1 + \norme{W}_{L^{\infty}} + \norme{V}_{L^{\infty}}^{1/2}\Big) R \Bigg) \int_{R<|x|< 2 R} |u(x)| d x.
\end{equation}
\end{prop}
Before continuing, let us make some comments on \Cref{cor:Landis}.
\begin{itemize}
\item 
Note that \Cref{cor:Landis} directly implies the qualitative Landis conjecture for solution $u$ verifying \eqref{eq:ellipticu} and satisfying the exponential decay estimate 
\begin{equation}
\label{eq:expdecay}
\exists C, \varepsilon > 0,\ \forall x \in \R^N,\ |u(x)| \leq C \exp(-|x|^{1+\varepsilon}).
\end{equation}
Indeed, in this case, \eqref{eq:expdecay} implies that the right hand side term of \eqref{eq:quantitativeL1} goes to $0$ as $R \rightarrow + \infty$, so
\begin{equation*}
u \equiv 0.
\end{equation*} 
This result was already known from \cite{Ros18} or \cite{ABG17} because the generalized principal eigenvalue of the elliptic operator $\mathcal{L} u = -\Delta u - \nabla \cdot (W u) + V u$ is positive. Indeed, the adjoint operator of $\mathcal{L}$ satisfies the maximum principle so the generalized principal eigenvalue is positive: $\lambda_1(\mathcal{L}) = \lambda_1(\mathcal{L}^*) >0$. The Landis conjecture is still open nowadays without sign condition on the real-valued potential $V$, see \Cref{conj:Landis}.
\item 
Note that \Cref{cor:Landis} extends, in a (very) weak sense, quantitative unique continuation \eqref{eq:quantitativeContinuation} from \cite{KSW15} to any spatial dimension $N \geq 1$. Indeed, for a non-trivial solution $u$ to \eqref{eq:ellipticu}, we can assume for instance that $\int_{|x| < 1} |u(x)| dx = 1$ so \eqref{eq:quantitativeL1} implies
$$ \forall R \geq 1,\ \sup_{R<|x|< 2 R} |u(x)| \geq \exp (- C R),$$
for a positive constant $C$ depending on $N$, $\norme{W}_{\infty}$, $\norme{V}_{\infty}$.
\end{itemize}

We prove \Cref{cor:Landis} as a consequence of \Cref{th:mainresult1} and the crucial idea of the proof of \Cref{th:LandisOneD}.

\begin{proof}
\indent For $R>0$, we denote by $\phi_R$ the solution to \eqref{eq:ellipticphi}, posed on the ball $\Omega=B_0(2R)$, with homogeneous Dirichlet boundary conditions and $F = \sign(u)$. This elliptic problem is well-posed according to \cite[Chapter 8, Section 8.2]{GT83}.\\
\indent Then, from \Cref{th:mainresult1}, we deduce the following exponential growth for $\phi$, 
\begin{equation}
\label{eq:expgrowthProof}
\forall |x| \leq 2R,\ |\nabla \phi_R(x)| \leq \exp\Bigg(C\Big(1 + \norme{W}_{L^{\infty}} + \norme{V}_{L^{\infty}}^{1/2}\Big) R \Bigg).
\end{equation}
Let $\chi \in C_c^{\infty}(B_0(2R))$ be a cut-off function such that $\chi \equiv 1$ in $B_0(R)$ and $\chi \equiv 0$ for $|x| > 3/2 R$. By multiplying \eqref{eq:ellipticu} by $\chi \phi_R$ then integrating in $B_0(2R)$ and using integration by parts, we find
$$\int_{B_0(2R)} \chi |u| \leq C \norme{\phi_R}_{W_0^{1, \infty}(B_0(2R))} \int_{R < |x|< 2R}  |u|,$$ which along with \eqref{eq:expgrowthProof} gives \eqref{eq:quantitativeL1}.
\end{proof}
We can also prove the following result.
\begin{prop}
\label{cor:LandisBis}
Let $W,V,u$ be as in \Cref{cor:Landis}.
Then there exists a universal positive constant $C>0$ such that 
\begin{equation}
\label{eq:quantitativeL1Bis}
\forall R > 0,\ \int_{|x| < R} |u(x)| dx \leq \exp\Bigg(C\Big(1 + \norme{W}_{L^{\infty}} \Big) R \Bigg) \int_{|\sigma|=R} |u(\sigma)| d \sigma.
\end{equation}
\end{prop}
Note that \eqref{eq:quantitativeL1Bis} does not depend on $\norme{V}_{L^{\infty}}$.\\
\indent The following proof relies on the same duality argument as before but with another estimate of the function $\phi_R$ on the boundary. This provides a better estimate compared to \Cref{cor:Landis}.
\begin{proof}
By a simple density argument, we can assume without loss of generality that $W \in C(\R^N;\R^N)$.\\
\indent For $R>0$, we denote by $\phi_R$ the solution to \eqref{eq:ellipticphi}, posed on the ball $\Omega=B_0(R)$, with homogeneous Dirichlet boundary conditions and $F = \sign(u)$. This elliptic problem is well-posed according to \cite[Chapter 8, Section 8.2]{GT83}.\\
\indent We claim that we have the following bound
\begin{equation}
\label{eq:expgrowthProofBis}
\forall |\sigma| = R,\ |\partial_{\nu} \phi_R(\sigma)| \leq \exp\Bigg(C\Big(1 + \norme{W}_{L^{\infty}} \Big) R \Bigg).
\end{equation}
Admit \eqref{eq:expgrowthProofBis} for the moment, let us multiply \eqref{eq:ellipticphi} by $u$, then integrate in $\Omega$, so after integration by parts, we obtain
\begin{equation}\label{eq:proofLandisInt}
\int_{|x| < R} |u(x)| dx = - \int_{|\sigma|=R} \partial_{\nu} \phi_R(\sigma) u(\sigma) d \sigma + \int_{|\sigma|=R} \phi_R(\sigma) (Wu)(\sigma) \cdot \nu d \sigma.
\end{equation}
Plugging the bound \eqref{eq:expgrowthProof} in the identity \eqref{eq:proofLandisInt}, we obtain
$$\int_{|x| < R} |u(x)| dx\leq \exp\Bigg(C\Big(1 + \norme{W}_{L^{\infty}} \Big) R \Bigg) \int_{|\sigma|=R} | u(\sigma) | d \sigma.$$
This is exactly the expected bound \eqref{eq:quantitativeL1Bis}.\\
\indent To prove the estimate \eqref{eq:expgrowthProofBis}, we exhibit a supersolution to \eqref{eq:ellipticphi}. For some $a>0$, let $\varphi$ be the following radial function 
\begin{equation*}
\varphi(x) = \varphi(r) := \frac{e^{aR} - e^{ar}}{a},\ \text{where}\ r := |x|.
\end{equation*}
Then for $N \geq 2$, we have
\begin{equation*}
\varphi'(r) = - e^{ar},\ \varphi''(r) = -a e^{ar},\ \nabla \varphi = - e^{ar} \frac{x}{r},\ \Delta \varphi = - e^{ar} \left(a + \frac{N-1}{r}\right).
\end{equation*}
So 
\begin{equation}
\label{eq:supersolutionvarphi}
-\Delta \varphi + W \cdot \nabla \varphi + V \varphi = e^{ar} \left(a +  \frac{N-1}{r} + W \cdot \frac{x}{r} \right) + V \frac{e^{aR} - e^{ar}}{a}.
\end{equation}
So $\varphi$ is a supersolution to \eqref{eq:ellipticphi} with $F= \sign(u)$ if the right hand side of \eqref{eq:supersolutionvarphi} is bigger than $1$, which is the case for $a \geq - W \cdot \frac{x}{r} + e^{-ar}$ hence for $a := \norme{W}_{\infty} + 1$.\\
\indent We deduce that $\phi \leq \varphi$. By using the homogeneous Dirichlet boundary conditions for $\phi$ and $\varphi$, we deduce that for $|\sigma| = R$, we have
$$ \phi \leq \varphi \Rightarrow \partial_{\nu} \varphi(\sigma) \leq \partial_{\nu} \phi(\sigma) \Rightarrow - \partial_{\nu} \phi(\sigma) \leq - \partial_{\nu} \varphi(\sigma) = e^{aR}.$$
By the same kind of computations, we show that $\phi \geq - \varphi$ to get $\partial_{\nu} \phi(\sigma) \leq -  \partial_{\nu} \varphi(\sigma) = e^{aR} $. Therefore, we have \eqref{eq:expgrowthProofBis}.
%\indent The bound \eqref{eq:EstimationCouronne} can be obtained for instance from Cacciopoli's inequality. We multiply \eqref{eq:ellipticu} by $\theta^2 u$ where $\theta$ is some cut-off function localized in the annulus and we integrate by parts, we find
%\begin{equation}
%\int_{\R^N} \theta^2 |\nabla u|^2 + \int_{\R^N} \nabla u \cdot (2 \theta \nabla \theta) u - \int_{\R^N} \left(u^2 W \cdot  (2 \theta \nabla \theta)  + u W \cdot (\theta^2 \nabla u)\right).
%\end{equation}
%So by Young's inequalities, we obtain
%\begin{equation}
%\int_{\R^N} \theta^2 |\nabla u|^2 \leq C \left(\int_{\R^N} |\nabla \theta| u^2 + \int_{\R^N} |\nabla \theta| |\theta| u^2 + \int_{\R^N} \theta^2 u^2\right),
%\end{equation}
%and from \eqref{eq:expdecay} we deduce that 
%\begin{equation}
%\int_{\R^N} \theta^2 |\nabla u|^2 \leq C \exp(-|R|^{1+ \varepsilon}).
%\end{equation}
%This provides the decreasing of the $H^1$-norm of $u$ in the annulus. To obtain the decreasing of the $H^2$-norm of $u$ in the annulus, we can use the equation \eqref{eq:ellipticu} to see that $\Delta u$ exponentially decreases in the annulus the all the $H^2$-norm. More elegant or easy proof?
\end{proof}

\appendix

\section{Proof of the gradient log estimate for the multiplier}
\label{sec:appendixa}

This section is devoted to the proof of \Cref{lem:kenig}. Observe that \Cref{lem:kenig} is similar to \cite[Lemma 2.2]{KSW15} but the proof is different. It has been communicated by Carlos Kenig.
\begin{proof}[Proof of \Cref{lem:kenig}]
Recalling \eqref{eq:ellipticpsi}, we have
\begin{equation}
\label{eq:GammaRtildeV}
-\Delta \tilde{\psi} + \tilde{W} \cdot \nabla \tilde{\psi} +  \tilde{V}  \tilde{\psi}= 0\ \text{in}\ B_0(2),
\end{equation}
where $\tilde{W} = R W(Rx)$ and $\tilde{V} = R^2 V(Rx)$. Observe that 
\begin{equation}
\label{eq:estimationtildeW}
\forall x \in B_2,\ |\tilde{W}(x)| \leq R K,
\end{equation}
and
\begin{equation}
\label{eq:estimationtildeV}
\forall x \in B_2,\ |\tilde{V}(x)| \leq R^2 M.
\end{equation}
In the proof, the positive constants $C>0$ only depend on the spatial dimension $N$ and can vary from line to another. We split the proof in two steps.\\
\indent \textbf{Step 1: $L^{\infty}$- gradient estimate on a small ball.}\\
For every $x_0 \in B_0(2)$ and $r>0$ sufficiently small, we have
\begin{equation}
\label{eq:gradsmallpsi}
\norme{\nabla \tilde{\psi}}_{L^{\infty}(B_{x_0}(r))} \leq \frac{C}{r} \norme{\tilde{\psi}}_{L^{\infty}(B_{x_0}(2r))} + C \norme{\tilde{V}}_{\infty} r \norme{\tilde{\psi}}_{L^{\infty}(B_{x_0}(2r))}.
\end{equation}
\indent To obtain \eqref{eq:gradsmallpsi}, we proceed as follows. We use \cite[Chapter 3, Section 3.4, Theorem 3.9]{GT83} with $\Omega = B_{x_0}(2r)$, $f = \tilde{W} \cdot \nabla \tilde{\psi} +  \tilde{V}  \tilde{\psi}$, $d_x = d(x, B_{x_0}(2r))$,
\begin{align*}
\sup_{x \in B_{x_0}(2r)} d_x  |\nabla \tilde{\psi}(x)| \leq C \Bigg( \sup_{x \in B_{x_0}(2r)} |\tilde{\psi}(x)| + \sup_{x \in B_{x_0}(2r)} d_x ^2\Big( |\tilde{W}(x)|  |\nabla \tilde{\psi}(x)| + |\tilde{V}(x)|  | \tilde{\psi}(x)|\Big)\Bigg).
\end{align*}
So by taking $r$ sufficiently small, i.e. $r \leq \norme{\tilde{W}}_{\infty}/(2C)$, we obtain
\begin{equation}
\label{eq:gt83poisson}
\sup_{x \in B_{x_0}(2r)} d_x  |\nabla \tilde{\psi}(x)| \leq C \Bigg( \sup_{x \in B_{x_0}(2r)} |\tilde{\psi}(x)| + \norme{\tilde{V}}_{\infty} r^2 \sup_{x \in B_{x_0}(2r)} | \tilde{\psi}(x)|\Bigg).
\end{equation}
Moreover, we have 
\begin{equation}
\label{eq:smallballbigball}
\norme{\nabla \tilde{\psi}}_{L^{\infty}(B_{x_0}(r))} \leq \sup_{x \in B_{x_0}(r)} \frac{d_x}{d_x}  |\nabla \tilde{\psi}(x)|  \leq \frac{1}{r} \sup_{x \in B_{x_0}(r)} d_x  |\nabla \tilde{\psi}(x)| \leq \frac{1}{r} \sup_{x \in B_{x_0}(2r)} d_x  |\nabla \tilde{\psi}(x)|.
\end{equation}
Then, plugging \eqref{eq:smallballbigball} in \eqref{eq:gt83poisson}, we get \eqref{eq:gradsmallpsi}.

\indent \textbf{Step 2: Harnack's inequalities.}\\
Recalling that $\tilde{\psi}$ is nonnegative, we can use Harnack's inequality \cite[Chapter 8, Section 8.8, Theorem 8,20]{GT83} to obtain that 
\begin{equation}
\label{eq:harnackinequalityprecise}
\norme{\tilde{\psi}}_{L^{\infty}(B_{x_0}(2r))} \leq C \tilde{\psi}(x_0) \leq C \tilde{\psi}(x),\ \forall x \in B_{x_0}(r),
\end{equation}
for
\begin{equation}
\label{eq:smallr}
r = \frac{1}{C\left(\norme{\tilde{W}}_{\infty} + \norme{\tilde{V}}_{\infty}^{1/2}\right)}.
\end{equation}
So by putting together \eqref{eq:gradsmallpsi} and \eqref{eq:harnackinequalityprecise} for $r$ as in \eqref{eq:smallr}, we obtain that 
\begin{equation}
\label{eq:gradsmallpsiFinal}
\norme{\nabla \tilde{\psi}}_{L^{\infty}(B_{x_0}(r))} \leq C\left(\norme{\tilde{W}}_{\infty} + \norme{\tilde{V}}_{\infty}^{1/2}\right)\tilde{\psi}(x),\ \forall x \in B_{x_0}(r),
\end{equation}
so 
\begin{equation*}
\label{eq:gradsmallpsiFinalBis}
\norme{\frac{\nabla \tilde{\psi}}{\tilde{\psi}}}_{L^{\infty}(B_{x_0}(r))} \leq C\left(\norme{\tilde{W}}_{\infty} + \norme{\tilde{V}}_{\infty}^{1/2}\right).
\end{equation*}
Recalling that $x_0$ is taken arbitrary in $B_0(2)$, this is exactly the expected bound \eqref{eq:lemkenig} recalling \eqref{eq:estimationtildeW} and \eqref{eq:estimationtildeV}.
\end{proof}

\paragraph*{Acknowledgments.} The author deeply thanks Michel Pierre for many reasons: first, for communicating the proof of \Cref{th:LandisOneD}, which inspires this work, then for reading and making comments on a preliminary draft of this paper. The author also thanks Blair Davey and Carlos Kenig for stimulating discussions around \Cref{lem:kenig}.

\bibliographystyle{alpha}
\bibliography{Landis}

\end{document}